\documentclass[twoside,12pt]{article}
\usepackage{graphicx,amsmath,latexsym,amssymb,amsthm}
\usepackage[colorlinks=black,urlcolor=black]{hyperref}
at 9pt  

\date{}
\newtheorem{theorem}{Theorem}[section]

\newtheorem{definition}[theorem]{Definition}

\newtheorem{lemma}[theorem]{Lemma}

\newtheorem{remark}[theorem]{Remark}

\newtheorem{thm}{Theorem}[section]
\newtheorem{lem}[thm]{Lemma}
\newtheorem{cor}[thm]{Corollary}

\numberwithin{equation}{section}

\begin{document}
\setlength{\unitlength}{1cm}



\vskip1.5cm

 \centerline { \textbf{Asymptotics of  the eigenvalues and Abel basis property of the root     }}
 \centerline { \textbf{functions of new type Sturm-Liouville problems}}

\vskip.2cm


\vskip.8cm \centerline {\textbf{O. Sh. Mukhtarov$^{a,b}$, K.
Aydemir$^{c}$ and S. Y. Yakubov}$^{d}$}

\vskip.5cm

\centerline {$^a$Department of Mathematics, Faculty of Science,}
\centerline {Gaziosmanpa\c{s}a University,
 60250 Tokat, Turkey}
 \centerline {$^{b}$Institute of Mathematics and Mechanics, }
 \centerline { Azerbaijan National Academy of Sciences, Baku, Azerbaijan}
 \centerline {$^c$Department of Mathematics, Faculty of Art and Science,}
\centerline { Amasya University,05100, Amasya, Turkey} \centerline
{$^d$Department of Mathematics and Computer Science,,} \centerline {
University of Haifa, Haifa 31905, Israel}
 \centerline
{e-mail : {\tt omukhtarov@yahoo.com, kadriyeaydemr@gmail.com,
 }}
 \centerline{ rsmaf06@mathcs.haifa.ac.il }


\vskip.5cm \hskip-.5cm{\small{\bf Abstract :}This work investigates
spectrum and root functions(that is, eigen- and associated
functions) of a Sturm-Liouville problem involving an abstract linear
operator(nonselfadjoint in general) in the equation together with
supplementary transmission conditions at the one interior singular
point. So, the problem under consideration is not pure differential
problem. At first we establish isomorphism and coerciveness with
respect to the spectral parameter for the corresponding
nonhomogeneous problem. Then by suggesting an own our method we
prove that the spectrum of the considered problem is discrete and
derive an asymptotic approximation formulas for the eigenvalues. We
must note that Asymptotics of the eigenvalues of such type problems
is investigated at first in literature in the present work and the
obtained results are new even  in the continuous  case(i.e. without
transmission conditions). Finally it is shown that the system of
root functions form an Abel basis of the corresponding Hilbert
space.


\section{\textbf{Introduction}}
Self-adjoint boundary value problems(BVPs, for short) are of
significant importance in many models of applied mathematics and
quantum mechanics in spherical and cylindrical geometries. Among
these BVPs, the Sturm-Liouville problems is a typical one. Many
physical processes, such as the vibration of strings, the
interaction of atomic particles, electrodynamics of complex medium,
aerodynamics, polymer rheology  or the earth's free oscillations
yields Sturm-Liouville eigenvalue problems(see, for example, \cite{
ges, kon, pe, pr, ro, she,tik} and references cited therein).
Generally, the separation of variables method was applied on the
two-order partial differential equation to obtain a Sturm-Liouville
problem for each independent variable. This method is a cornerstone
in the study of partial differential equations, and is a major
element in physical problems. For example, consider a boundary value
problem for the one-dimensional wave equation
 \begin{eqnarray*}
&&\rho_0u_{tt}=(ku_{x})_{x}, \ \ \ 0\leq x \leq L, \  \nonumber  \\
&&u(0,t)=u(L,t)=0, \ \
\end{eqnarray*}
for the longitudinal displacement $u(x; t)$ of a string of length
$L$ with mass-density $\rho_0(x)$ and stiffness $k(x)$, both of
which we assume are smooth, strictly positive functions on $0\leq x
\leq L$. Looking for separable time-periodic solutions  we get
Sturm-Liouville problem
 \begin{eqnarray*}
-(k\varphi')'=\lambda \rho_0\varphi, \ \ \varphi(0)=\varphi(L)=0.
\end{eqnarray*}
Sturm Liouville theory was developed collaboratively by
Charles-Franciois Sturm (1803- 1855) and Joseph Liouville
(1809-1882) in order to generalise a relatively disorganised array
of second order linear differential equations used to model physical
problems. These included Bernoulli's work on vibrating strings and
Liouville's own work on heat conduction \cite{lut}. In 1910 Hermann
Weyl \cite{wey} gave the first rigorous treatment, in the case of an
equation of Sturm-Liouville type, of cases where continuous spectra
can occur. The theory was particularly significant because it
provided the first qualitative theory of differential equations, and
was thus very useful for solutions that could not be solved
explicitly. These problems involve self-adjoint (differential)
operators which play an important role in the spectral theory of
linear operators and the existence of the eigenfunctions.  The
development of classical, rather than the operatoric,
Sturm-Liouville theory in the years after 1950 can be found in
various sources; in particular in the texts of Atkinson \cite{at},
Coddington and Levinson \cite{co}, Levitan and Sargsjan \cite{le}
and Naimark \cite{na}.  Spectral problems associated with
differential operators having only a discrete spectrum and depending
polynomially on the spectral parameter have been considered by
Gohberg and Krein \cite{goh}, and by Keldysh \cite{kel}. They
studied the spectrum and principal functions of such problems and
showed the completeness of the principal functions in the
corresponding Hilbert function space. There are a lot of studies
about the spectrum of such operators \cite{ag,ak}.  For the
background and applications of the boundary value problems to
different areas, we refer the reader to the monographs and some
recent contributions as \cite{ ag,  al2, e.t1, ka2, eba1, bi1, bo,
ch, ism, kan, ker, mak,  os1, ka2y, titc}. Note that in recent
years, there has been growing interest in boundary- value problems
with interior singularities(see, for example, \cite{ka2, eba1, bo,
ch, fatma, os1, tit} and references cited therein).

 In this study we
shall investigate some spectral aspects of a new type
Sturm-Liouville equation involving an abstract linear operator $A$,
namely  the "differential" equation
\begin{equation}\label{1}
Lu\equiv p(x)u^{\prime \prime }+ Au=\lambda u \ \
x\in[-1,0)\cup(0,1]
\end{equation}
together with boundary conditions at the end-points $x=-1,1$ given
by
\begin{equation}\label{2}
L_1u\equiv\alpha_0 u(-1)+\alpha_1 u'(-1)=0
\end{equation}
\begin{equation}\label{3}
L_2u\equiv\beta_0 u(1)+\beta_1 u'(1)=0
\end{equation}
and the transmission conditions at the point of singularity $x=0$
given by
\begin{equation}\label{4}
u'(0^{-})=\gamma_{0}u(0^{-})+\delta_{0}u(0^{+})=0
\end{equation}
\begin{equation}\label{5}
 u'(0^{+})=\gamma_{1}u(0^{-})+\delta_{1}u(0^{+})=0
\end{equation}
where $p(x)=p_1$ for $x\in [-1,0),$ $p(x)=p_2$ for $x\in (0,1]$;
$p_{i}, \alpha_{i}, \beta_{ij}, \delta_{i}, \gamma_{i} \ (i=0,1)$
are real numbers; $p_1\neq0, p_2\neq0,$ $|\alpha_0|+|\alpha_1|\neq0,
\ |\beta_0|+|\beta_1|\neq0;$ $\lambda$ is a complex spectral
parameter.   Transmission problems appear frequently in various
fields of physics and technics \cite{tik,tit,voito}. For instance,
in electrostatics and magnetostatics the model problem which
describes the heat transfer through an infinitely conductive layer
is a transmission problem (see, \cite{pha} and the references listed
therein). Another completely different field is that of "hydraulic
fracturing" (see, \cite{can}) used in order to increase the flow of
oil from a reservoir into a producing oil well. Further examples can
be found in Dautray and Lions \cite{dau,lan}.

\section{Some auxiliary facts and results}
 Let $E$ and $F$ be two Banach spaces, for which a set-theoretical
inclusion $E\subset F$ holds, and the linear space $F$ induces on
$E$ the linear structure coinciding with the structure of the linear
space $E$ and let $J$ be the embedding operator from $E$ to $F$,
i.e. $Jx=x$ for all $x\in E$. If this operator is continuous, we say
that the embedding $E\subset F$ is continuous. Similarly if the
operator $J$ is compact then the embedding  $E\subset F$ is said to
be a compact embedding. If $\overline{E}= F$, i.e. the subset
$E\subset F$ is dense in the Banach space $F$, then we say that the
embedding $E\subset F$ is dense.

Throughout in this study, the notation of inclusion "$\subset $"
must be understood in the set-theoretical and in the topological
meaning. Let $E_1$ and $E_2$ be two complex Banach spaces, both
linearly and continuously embedded in a Banach space $E$. Then the
pair $\{E_1,E_2\}$ are said to be an interpolation couple. Let us
define the linear space $E_1+E_2$ by
$$E_1+E_2=\{u\in E|\textrm{there are}\ u_1\in E_1 \ \textrm{and}\ u_2\in E_2 \ \textrm{such that} \ u=u_1+u_2 \}$$
It is known that this linear space forms the Banach space
(see,\cite{trib}) with respect to the norm given by
$$\|u\|_{E_1+E_2}=\inf\{\|u_1\|_{E_1}+\|u_2\|_{E_2}|\ u_1\in E_1, u_2\in E_2\ \ u_1+u_2=u\}$$
where the infimum is taken over all representations $u= u_1+u_2$ in
the described way. It is easy to  see that for any $t>0$ the
functional $K(t,u)$ defined on $E_1+E_2$ by
$$K(t,u)=\inf\{\|u_1\|_{E_1}+t\|u_2\|_{E_2}|\ u_1\in E_1, u_2\in E_2\ \ u_1+u_2=u\}$$
is an equivalent norm in the Banach space $E_1+E_2$. An
interpolation space for interpolation couple $\{E_1,E_2\}$ by
$K$-method is defined  as follows(see \cite{trib})
 $$\{E_1,E_2\}_{\theta,p}:=\{u:u\in E_1+E_, \ \|u\|_{\{E_1,E_2\}}:=(\int\limits_{0}^{\infty}\frac{K^{p}(t,u)}{t^{1+\theta p}}dt)^{\frac{1}{p}}<\infty \}$$
Due to  [\cite{trib}, Triebel 1.3.3]  there exists a positive number
$C_{\theta,p}$, such that for all $u\in E_1\cap E_2$
\begin{eqnarray}\label{6ks1}\|u\|_{(E_1,E_2)_{\theta,p}}\leq C_{\theta,p}\|u\|_{E_1}^{1-\theta}\|u\|_{E_2}^{\theta}.\end{eqnarray}
Applying the well-known Young inequality to the right hand of the
last inequality we have that for each $\epsilon>0$ there exists
$C(\epsilon)>0$ such that for all $u\in E_0\cap E_1$
\begin{eqnarray}\label{6l1}\|u\|_{(E_1,E_2)_{\theta,p}}\leq \epsilon\|u\|_{E_1}+C(\epsilon)\|u\|_{E_2}.\end{eqnarray}
\begin{definition}The Sobolev space $W_{2}^{n}(a,b)(n=0,1,2,...)$ is
the Hilbert space consisting of all functions $f\in L_{2}(a,b)$ that
have square-integrable generalized derivatives  $f',f'',...,f^{(n)}$
on $(a,b)$ with the inner-product
$$\langle f,g\rangle_{W_{2}^{n}(a,b)}=\sum_{k=0}^{n}\langle f^{(k)},g^{(k)}\rangle_{L_{2}(a,b)}.$$
where that under $W_{2}^{0}(a,b)$ we mean $L_{2}[a,b].$
\end{definition}
\begin{definition}Let [a,b] be any finite interval,  $0<s\neq integer$ be any real number and let
$n$ be any integer such that $n>s$. For such $s$ the interpolation
space $W_{2}^{s}(a,b)$ is defined as
\begin{eqnarray}\label{6k1}W_{2}^{s}(a,b):=(W_{2}^{n}(a,b),L_{2}(a,b))_{1-\frac{s}{n},2}\end{eqnarray}
\end{definition}
\begin{remark}It is known that(see,for example \cite{trib}) the
equality (\ref{6k1}) is hold even in the case when $s$ is also
integer.
\end{remark}
Below we shall use the direct sum of Sobolev spaces $W_{2}^{s}(-1,0)
\oplus W_{2}^{s}(0,1)$ of functions on $(-1,0)\cup(0,1)$ belonging
to $W_{2}^{s}(-1,0)$ and  $W_{2}^{s}(0,1)$ in $(-1,0)$ and $(0,1)$
respectively, with the norm $$\|f\|_{W_{2}^{s}}:=
(\|f\|^{2}_{W_{2}^{s}(-1,0)}+\|f\|^{2}_{W_{2}^{s}(0,1)})^{1/2}.$$
 From
inequalities (\ref{6ks1}) and (\ref{6k1}) we have the following
Lemma.
\begin{lem}\label{lenm} Let $0\leq s\leq2$.
Then there is a constant $C>0$ such that for all $u\in
W_{2}^{2}(-1,0) \oplus W_{2}^{2}(0,1)$ and $\lambda\in\mathbb{C}$
the following inequality holds:
\begin{eqnarray}\label{6h}|\lambda|^{2-s}\|u\|_{W_{2}^{s}}\leq C(\|u\|_{W_{2}^{2}}+|\lambda|^{2}\|u\|_{L_{2}})\end{eqnarray}
\end{lem}
By using (\ref{6l1}) and (\ref{6k1}) we have
\begin{lem}\label{lesnm} Let $k\geq0$ any real number.
Then for each $\epsilon>0$ there is a constant $C(\epsilon)>0$ such
that for all $u\in W_{2}^{k+\frac{1}{2}}(-1,0) \oplus
W_{2}^{k+\frac{1}{2}}(0,1)$ the following inequality  holds
\begin{eqnarray}\label{6hl}\|u\|_{W_{2}^{k}}\leq \epsilon\|u\|_{W_{2}^{k+\frac{1}{2}}}+C(\epsilon)\|u\|_{L_{2}}\end{eqnarray}
\end{lem}
\begin{lem}\label{lesnmt} Let the following conditions be satisfied:

i)$H_1$ and $H_2$ are separable Hilbert spaces and $H_1\subset H_2$

ii)The embedding operator  $J:H_1\rightarrow H_2$ is densely defined
and continuous.

iii)The  operator  $B$ acts compactly from $H_1$ into $H_2.$ Then
for any $\epsilon>0$ there exist a constant $C(\epsilon)>0$ such
that
\begin{eqnarray}\label{6hhl}\|Bu\|_{H_{2}}\leq \epsilon\|u\|_{H_{1}}+C(\epsilon)\|u\|_{H_{2}}\end{eqnarray}
for all $u\in H.$
\end{lem}
\begin{proof} The proof follows immediately from [\cite{ka2y}, Lemma
1.2.8/3].
\end{proof}
Denoting $$L_3u=u'(0^{-})-\gamma_{0}u(0^{-})-\delta_{0}u(0^{+})=0$$
and
$$L_4u= u'(0^{+})-\gamma_{1}u(0^{-})-\delta_{1}u(0^{+})=0$$  we shall
define the operator $\pounds_0$ in the Hilbert space $L_{2}(-1,0)
\oplus L_{2}(0,1)$ by domain of definition
\begin{eqnarray}\label{6}
D(\pounds_0)=&\big \{&u| u\in W_{2}^{2}(-1,0) \oplus W_{2}^{2}(0,1),
L_{\nu}u=0, \ \nu=1\div4 \big \}
\end{eqnarray}
and action low $\pounds_0 u=p(x)u''.$ Throughout in below we shall
assume that $D(A)\supset D(\pounds_0)$ and define the operator
$\pounds$ in the Hilbert space $L_{2}(-1,0) \oplus L_{2}(0,1)$ by
domain of definition $D(\pounds)= D(\pounds_0)$ and action low
$$\pounds(u)=p(x) u''+Au.$$
 \begin{remark}\label{yukl}Note that under spectrum and root
functions of the problem $(\ref{1})-(\ref{5})$ we mean the spectrum
and root functions of the operator $\pounds$,
respectively.\end{remark}
\begin{theorem}\label{simm}
If $p_1\delta_0=a_2\gamma_{1}$, then the operator $\pounds_0$ is
densely defined and symmetric.
\end{theorem}
\begin{proof}
Denote by $C_{0}^{\infty}[-1,0) \oplus C_{0}^{\infty}(0,1]$ the set
of infinitely differentiable functions in $[-1,0)\cup(0,1],$ each of
which vanishes on some neighborhoods of the points $x=-1, x=0$ and
$x=1.$ It is well-known that this set is dense in the Hilbert space
$L_{2}(-1,0) \oplus L_{2}(0,1).$ Since $C_{0}^{\infty}[-1,0) \oplus
C_{0}^{\infty}(0,1]\subset D(\pounds_0)$ we have that $D(\pounds_0)$
is also dense in the Hilbert space $L_{2}(-1,0) \oplus L_{2}(0,1).$
Further, by using Lagrange's formula we can derive easily that
\begin{eqnarray}\label{10}
(\pounds_0u,v)_{L_2} =(u,\pounds_0v)_{L_2}
+p_1W(u,\overline{v};x)|_{-1}^{0^{-}} +
p_2W(u,\overline{v};x)|_{1}^{0^{+}}.
\end{eqnarray}
Since both $u$ and $\overline{v}$ satisfy the boundary conditions
$(\ref{2})-(\ref{3})$ it follows that
\begin{eqnarray}\label{plo}
W(u,\overline{v};-1)=W(u,\overline{v};1)=0.
\end{eqnarray}
Furthermore, from the transmission conditions $(\ref{4})-(\ref{5})$
we have
\begin{eqnarray}\label{19}
&&p_1W(u,\overline{v};0^{-})+ p_2W(u,\overline{v};0^{+}) =
-p_1(\gamma_0)u(0^{-})+\delta_0u(0^{+})\overline{v}(0^{-})
\nonumber\\&&+p_1(\gamma_0)\overline{v}(0^{-})+\delta_0\overline{v}(0^{+})u(0^{-})
-p_2(\gamma_1)u(0^{-})+\delta_1u(0^{+})\overline{v}(0^{+})
\nonumber\\&&+p_2(\gamma_1)\overline{v}(0^{-})+\delta_1\overline{v}(0^{+})u(0^{+})
\nonumber\\&&=-(p_1\delta_{0}-p_2\gamma_{1})(u(0^{+})\overline{v}(0^{-})-u(0^{-})\overline{v}(0^{+}))=0.
\end{eqnarray}
Putting (\ref{plo}) and (\ref{19}) in (\ref{10}) yields
\begin{eqnarray}\label{20}
(\pounds_0u,v)_{L_2} =(u,\pounds_0v)_{L_2},
\end{eqnarray}
for all  $u,v \in D(\pounds_0).$ The proof is complete.
\end{proof}
\section{Separation results for the corresponding nonhomogeneous problem}
Consider the following nonhomogeneous problem
\begin{equation}\label{gyh}
Lu-\lambda u(x)=f(x), \ \ \ \ L_{v}u=f_v, \ \ \ \ v=1\div4.
\end{equation}
for arbitrary $f\in L_{2}(-1,0) \oplus L_{2}(0,1), \
f_v\in\mathbb{C}, \ v=1\div4.$ The next theorem is crucial for
further consideration.

\begin{theorem}\label{ghy}Let the following conditions be satisfied:

1. $p_1\delta_{0}=p_2\gamma_{1}$

2. The operator $A$  acts compactly from the Hilbert space
$W_{2}^{2}(-1,0) \oplus W_{2}^{2}(0,1)$ to the Hilbert space
$L_{2}(-1,0) \oplus L_{2}(0,1)$.

Then for any $\epsilon>0$ (small enough) there exists $R_\epsilon>0$
such that for all $\lambda\in\mathbb{C}$ satisfying
$|\arg\lambda\pm\frac{\pi}{2}|>\epsilon$, $|\lambda|>R_\epsilon$,
the operator $\Im(\lambda):u\rightarrow((\pounds-\lambda
I)u,L_{1}u,L_{2}u,L_{3}u,L_{4}u)$ is an isomorphism from
$W_{2}^{2}(-1,0) \oplus W_{2}^{2}(0,1)$ onto  $L_{2}(-1,0) \oplus
L_{2}(0,1)\oplus\mathbb{C}^{4}$  and  the following coercive
estimate holds
\begin{eqnarray}\label{34hýo}|\lambda|\|u\|_{L_{2}}+|\lambda|^{\frac{1}{2}}\|u\|_{W_{2}^{1}}+\|u\|_{W_{2}^{2}}&\leq& C(\epsilon)(\|f\|_{L_{2}}+|\lambda|^{\frac{1}{4}}(|f_1|+|f_2|\nonumber\\&+&|f_3|+|f_4|))\end{eqnarray}
where $C(\epsilon)$ is the constant which is dependent only on
$\epsilon>0$
\end{theorem}
\begin{proof}Let us define the linear functionals
$\ell_i(u)(i=1\div4)$ by the equalities $\ell_1u=\alpha_0u(-1),\
\ell_2u=\beta_0 u(1),\
\ell_3u=-\gamma_{0}u(0^{-})-\delta_{0}u(0^{-})$ and
$\ell_4u=-\gamma_{1}u(0^{-})-\delta_{1}u(0^{+}).$ Let $u\in
W_{2}^{2}(-1,0) \oplus W_{2}^{2}(0,1)$ be any solution of the
problem (\ref{gyh}). Denoting $g(x)=f(x)-(Au)(x),$
$g_i=f_i-\ell_i(u)(i=1\div4)$ consider the differential equation
\begin{eqnarray}\label{57ok}p(x)u''-\lambda u=g(x),\ \ \ x\in[-1,0)\cup(0,1]\end{eqnarray}
together with boundary conditions
\begin{eqnarray}\label{57oks}\alpha_1u'(-1)=g_1,  \ \ \ \beta_1 u'(1)=g_2  \end{eqnarray}
and with transmission conditions
\begin{eqnarray}\label{57osk}u'(0^{-})=g_3,  \ \ \ u'(0^{+})=g_4  \end{eqnarray}
By applying the Theorem 3 in \cite{ka2y} to the problem
$(\ref{57ok})-(\ref{57osk})$ we obtain immediately the next a priori
estimate
\begin{eqnarray}\label{34h}
&&|\lambda|\|u\|_{L_{2}}+
|\lambda|^{\frac{1}{2}}\|u\|_{W_{2}^{1}}+\|u\|_{W_{2}^{2}}\leq
C(\epsilon)(\|g\|_{L_{2}}+|\lambda|^{\frac{1}{4}}\sum_{i=1}^{4}|g_i|)
\nonumber\\&&\leq
C(\epsilon)[(\|f\|_{L_{2}}+|\lambda|^{\frac{1}{4}}\sum_{i=1}^{4}|f_i|)+(\|Au\|_{L_{2}}+|\lambda|^{\frac{1}{4}}\sum_{i=1}^{4}|\ell_iu|)]
\end{eqnarray}
for the solution of problem (\ref{gyh}). Let us estimate the right
hand of this inequality. Since the operator $A$ is compact from
$W_{2}^{2}(-1,0) \oplus W_{2}^{2}(0,1)$ into $L_{2}(-1,0) \oplus
L_{2}(0,1)$, by virtue of Lemma \ref{lesnmt} for any $\delta>0$
there exist a constant $C(\delta)$ such that
\begin{eqnarray}\label{34hh}
\|Au\|_{L_{2}}\leq\delta\|u\|_{W_{2}^{2}}+C(\delta)\|u\|_{L_{2}}
\end{eqnarray}
 By virtue of  (\ref{34h}) and Lemma
\ref{lesnmt} it follows that, for any $\delta>0$ there is a constant
$C(\delta)>0$ such that
\begin{eqnarray}\label{3hh}
\|Au\|_{L_{2}}\leq(\delta+C(\delta)|\lambda|^{-1})(\|u\|_{W_{2}^{2}}+|\lambda|^{\frac{1}{2}}\|u\|_{L_{2}})
\end{eqnarray}
Since the embedding $C[a,b]\subset W_{2}^{1}[a,b]$  is continuous
for arbitrary finite interval $[a,b],$ then by virtue of Lemma
\ref{lesnm} we find that for each $\delta>0$ there exists a constant
$C=C(\delta)>0$ such that
\begin{eqnarray}\label{4hs}
|\ell_vu|\leq
C\|u\|_{W_{2}^{1}}\leq\delta\|u\|_{W_{2}^{\frac{3}{2}}}+C(\delta)\|u\|_{L_{2}}
\end{eqnarray}
for all $u\in W_{2}^{\frac{3}{2}}(-1,0) \oplus
W_{2}^{\frac{3}{2}}(0,1).$ Now applying Lemma \ref{lenm} we find
\begin{eqnarray}\label{4kh}
\|u\|_{
W_{2}^{\frac{3}{2}}}\leq|\lambda|^{-\frac{1}{4}}(\|u\|_{W_{2}^{2}}+|\lambda|\|u\|_{L_{2}}).
\end{eqnarray}
Putting this in the previous inequality gives us the inequality
\begin{eqnarray}\label{4yt}
|\lambda|^{\frac{1}{4}}|\ell_vu|\leq(\delta+C(\delta)|\lambda|^{-\frac{3}{4}})(\|u\|_{W_{2}^{2}}+|\lambda|\|u\|_{L_{2}})
\end{eqnarray}
Making use the inequalities (\ref{34h})-(\ref{4yt}) we get
\begin{eqnarray*}\label{4lk}
\sum_{k=0}^{2}|\lambda|^{1-\frac{k}{2}}\|u\|_{W_{2}^{k}}&\leq&
C(\epsilon)(\|f\|_{L_{2}}+|\lambda|^{\frac{1}{4}}\sum_{v=1}^{4}|f_v|)+C(\epsilon)(\delta+C(\delta)|\lambda|^{\frac{-1}{4}})(\|u\|_{W_{2}^{2}}\nonumber\\&+&|\lambda|\|u\|_{L_{2}})
\end{eqnarray*}
in the angle
$G_\epsilon^{\pm}:=\{\lambda\in\mathbb{C}\mid|arg\lambda\pm\frac{\pi}{2}|>\epsilon\}$
for sufficiently large $|\lambda|$. It is obvious that for any fixed
$\epsilon>0$ we can choose $\delta>0$ so small and $|\lambda|$ so
large that
$C(\epsilon)(\delta+C(\delta)|\lambda|^{\frac{-1}{4}})<1$.
Consequently, for $\lambda\in G_\epsilon^{\pm}$ sufficiently large
in modulus we obtain a priori estimate (\ref{34hýo}). From this
estimation it follows that for $\lambda\in G_\epsilon^{\pm},$
sufficiently large in modulus a solution of the problem (\ref{gyh})
in $ W_{2}^{2}(-1,0) \oplus W_{2}^{2}(0,1)$ is unique. Now by
applying the Theorem 2 from \cite{ka2y} we have that for such
$\lambda$ the operator $\Im(\lambda)$ from $ W_{2}^{2}(-1,0) \oplus
W_{2}^{2}(0,1)$ into $(L_{2}(-1,0) \oplus
L_{2}(0,1))\oplus\mathbb{C}^{4}$ is Fredholm i.e. the  range of
$\Im(\lambda)$ is closed subset of the space $(L_{2}(-1,0) \oplus
L_{2}(0,1))\oplus\mathbb{C}^{4}$ and $\dim\ker\pounds(\lambda)=\dim
co \ker\pounds(\lambda)<\infty.$ Consequently, the range of
$\Im(\lambda)$ coincide with the whole space $(L_{2}(-1,0) \oplus
L_{2}(0,1))\oplus\mathbb{C}^{4}$. From this and the fact that the
operator $\pounds(\lambda)$ is injective, the statement of the
theorem follows.
\end{proof}
\begin{cor}\label{selkf}Let the conditions of the previous theorem
be  satisfied. Then for any $\epsilon>0$ there exists $R_\epsilon>0$
such that all complex numbers $\lambda$ satisfying
$\mid\arg\lambda\pm\frac{\pi}{2}\mid >\epsilon$,
$\mid\lambda\mid>R_\epsilon$ are regular values of the operator
$\pounds$ and for the resolvent operator
$R(\lambda,\pounds):L_{2}(-1,0) \oplus L_{2}(0,1)\rightarrow
L_{2}(-1,0) \oplus L_{2}(0,1)$ the following inequality holds:
$$\|R(\lambda,\pounds)\|\leq C(\epsilon)|\lambda|^{-1}.$$
\end{cor}
\begin{proof}Putting $f_1=f_2=f_3=f_4$ in (\ref{34hýo}) we have, in
particular, that
$$|\lambda|\|u\|_{L_{2}}\leq C(\epsilon)\|f\|_{L_{2}},$$
that is, $$|\lambda|\|R(\lambda,\pounds)f\|_{L_{2}}\leq
C(\epsilon)\|f\|_{L_{2}}.$$ The proof is complete.
\end{proof}
\begin{cor}\label{selkff}Under conditions of the Theorem \ref{ghy} the  resolvent operator
$R(\lambda,\pounds)$ acts boundedly from $L_{2}(-1,0) \oplus
L_{2}(0,1)$ into $W_{2}^{2}(-1,0) \oplus W_{2}^{2}(0,1)$.  \end{cor}
\begin{proof}Again, $f_1=f_2=f_3=f_4$ in (\ref{34hýo}) we have, in
particular, that $$ \|u\|_{W_{2}^{2}}\leq C(\epsilon)\|f\|_{L_{2}}$$
so $$|\lambda|\|R(\lambda,\pounds)f\|_{W_{2}^{2}}\leq
C(\epsilon)\|f\|_{L_{2}}$$ which completes the proof.
\end{proof}
By using the above results we can prove the next result.
\begin{theorem}\label{self}
Let $p_1\delta_{0}=p_2\delta_{1}.$ Then the operator $\pounds_0$ is
self-adjoint.
\end{theorem}
\begin{proof}
By Theorem \ref{simm} that  the operator $\pounds_0$ is densely
defined and symmetric. Taking into account the corollary \ref{selkf}
and applying the familiar theorem of Functional Analysis about the
extensions of symmetric operators. (see \cite{trib}) it follows that
the operator $\pounds_0$ is closed symmetric operator and both index
defect of this operator is equal to zero, i.e. the operator
$\pounds_0$ is self-adjoint.
\end{proof}

\section{Discreteness of the spectrum and
asymptotic behaviour of eigenvalues}

At first let us give some needed definitions. Let $S$ be unbounded
closed linear operator in separable complex Hilbert space $H$ and
let $\lambda_0$ be any eigenvalue of this operator. Then the linear
manifold $M_{\lambda_0}:=\cup_{n=1}^{\infty}Ker(S-\lambda_0I)^{n}$
is called a root lineal of $S$ according to eigenvalue $\lambda_0$.
Elements of this lineal are called a root vectors of $S$. The
dimension of the lineal $M_{\lambda_0}$ is called an algebraic
multiplicity of the eigenvalue $\lambda_0.$ The spectrum of operator
$S$ is called discrete if whole spectrum $\sigma(S)$ consist only of
eigenvalues with finite multiplicity and the set of eigenvalues has
not finite limit point. For such operators by $N(r,S)$ we denote the
number of eigenvalues belonging to the ball $\mid\lambda\mid\leq r $
provided that each of eigenvalues counted according to their
algebraic  multiplicity. Let $\Phi$ be any subset of the complex
plane $\mathbb{C}$. Then by $N(r,\Phi,S)$ we denote the number of
eigenvalues of operator $S$ belonging to $\{\lambda\in
\mathbb{C}|\mid\lambda\mid\leq r\}\cap\Phi$ provided thateach of
eigenvalues counted according  to their algebraic multiplicity;
\begin{eqnarray*}\label{23}(N(r,\Phi,S)=
\sum_{\lambda_n\in\{\lambda:\mid\lambda\mid\leq r\}\cap\Phi\}} 1 \ \
\ )
\end{eqnarray*}
Denote $\psi_\alpha^{\pm}=\{\lambda:\mid\arg(\pm\lambda)\mid<
\alpha\}$, $R_{+}=\{x\in \mathbb{R}:x>0\}$ and  $R_{-}=\{x\in
\mathbb{R}:x<0\}$. If there is no danger of confusion, we shall
write  $N_{\pm}(r,\alpha,S)$ instead of
 $N(r,\psi_\alpha^{\pm},S)$  and
 $N_{\pm}(r,\alpha)$ instead of $N(r,R_{\pm},S)$. The operator $B$
is called $p$-subordinate(where $0\leq p\leq1$) to $S$ if its domain
$D(B)\supset D(S)$ and if there exist $b>0$ such that
\begin{eqnarray}\label{24}\|Bu\|\leq b\|Su\|^{p}\|u\|^{1-p}\ for all \
u\in D(S).
\end{eqnarray}
It is known that if $S$ is self-adjoint with discrete spectrum and
the operator $B$ is p-subordinate to $S (0\leq p<1)$, then the
spectrum of  $S+B$ is also discrete (see \cite{goh}, Lemma V.10.1).

\begin{lemma}\label{25} Let $S$ be self-adjoint with discrete
spectrum and let  $B$  is p-subordinate$(0\leq p<1)$ to $S $ and
$T=S+B$. Then the spectrum $T$ lies in the set $|Im\lambda|\leq b
|\lambda|^{p}$ (b is the same constant as in (\ref{24})) and for all
$\delta>0$ and $\alpha$ with $ 0< \alpha<\frac{\pi}{2}$, there are
$c_1>0$ and $c_2>0$ such that
\begin{eqnarray}\label{26}\|N_{\pm}(\tau,\alpha,T)-N_{\pm}(r,\alpha,S)\|&\leq&
c_1(N_{\pm}(\tau+b(1+\delta)\tau^{p},S)
\nonumber\\&-&N_{\pm}(\tau-b(1+\delta)\tau^{p},S)+c_2)
\end{eqnarray}
\end{lemma}
\begin{proof}
The proof  follows immediately  from the propositions [\cite{mar}
Lemma 2.1], Theorem \ref{lesnmt} and Remark \ref{yukl}
\end{proof}
Suppose that the operators has at least one regular point
$\lambda_0.$ Then the operator $B$ is called compact with respect to
the operator $S$ , if $D(B)\supset D(S)$ and if $BR(\lambda_0,S)$ is
compact. If $S$ is self-adjoint with discrete spectrum and $B$ is
p-subordinate$(0\leq p<1)$ to $S $ , then $B$ is compact with
respect to the operator $S$. It is known that, if $S$ is
self-adjoint with discrete spectrum and $B$ is compact with respect
to the operator  $S$, then the operator $S+B$ has also discrete
spectrum. (see, \cite{goh}, Lemma V.10.1)
\begin{lemma}\cite{mar}\label{27} Let $S$ be self-adjoint with discrete
spectrum and let  $B$ be compact with respect to  $S$ and $T=S+B$
Then, if the number of positive eigenvalues of the operator $S$ is
infinite and
\begin{eqnarray*}\label{28}
\lim_{\begin{array}{c}
        r \longrightarrow \infty \\
        \varepsilon \longrightarrow 0 \\
      \end{array}}
\ \frac{N_{+} \left(r(1 + \varepsilon),S \right)}{N_{+}(r, S)}
 &=&1.
\end{eqnarray*}\label{29}
then for each $\alpha( 0< \alpha<\frac{\pi}{2})$ the relation
\begin{eqnarray}
\lim_{r \longrightarrow \infty} \ \frac{N_{+} \left(r,\alpha,T
\right)}{N_{+}(r,S)} =1.
\end{eqnarray}
is hold
\end{lemma}

Now we shall derive asymptotic formulas for eigenvalues of the
problem (\ref{1})-(\ref{5}) for various type abstract operators $A$
appearing in the equation. In particular case, we shall prove that
there are infinitely many eigenvalues.
\begin{theorem}\label{30}
Let us satisfy the following  conditions

1. $p_1\delta_{0}=p_2\gamma_{1}.$

2. The operator $A$ acted boundedly from $W_{2}^{1}(-1,0) \oplus
W_{2}^{1}(0,1)$ to $L_{2}(-1,0) \oplus L_{2}(0,1)$, i.e. there is
$C>0$ such that $\|Au\|_{L_{2}}\leq C\|u\|_{W_{2}^{1}} \ \textrm{for
all} \ u\in W_{2}^{1}.$ Case1. If $p_1p_2<0$ then the eigenvalues of
(\ref{1})-(\ref{5}) can be arranged as one  two sequences
$\{\lambda_{n,1}\}_{1}^{\infty}$ and
$\{\lambda_{n,2}\}_{1}^{\infty}$ with asymptotic behaviour
\begin{equation}\label{31}
\lambda_{n,1}=-p_1\pi^{2} n^{2} + O(n), \ \
\lambda_{n,2}=-p_2\pi^{2} n^{2} + O(n)
\end{equation}
Case 2. If $p_1p_2>0$ then the eigenvalues of (\ref{1})-(\ref{5})
can be arranged  as   sequence $\{\lambda_{n}\}_{1}^{\infty}$ with
asymptotic behaviours
\begin{equation}\label{32}
\lambda_{n}= -\frac{p_1p_2}{p_1+2\sqrt{p_1p_2}+p_2}\pi^{2} n^{2} +
O(n). \ \
\end{equation}
\end{theorem}
\begin{proof}Since the embedding $W_{2}^{2}(-1,0) \oplus
W_{2}^{2}(0,1)\subset L_{2}(-1,0) \oplus L_{2}(0,1)$ is compact
(see, \cite{kato}) by virtue of corollary \ref{selkff} the resolvent
operator $R(\lambda,\pounds)$  is  compact in the space $L_{2}(-1,0)
\oplus L_{2}(0,1)$. Consequently spectrum of the operators
$\pounds_0$ and $\pounds$ are discrete. In view of Theorem
\ref{self} the operator $\pounds_0$ is  self-adjoint. At the other
hand, by applying the multiplicative inequality (\ref{6ks1}) we have
\begin{equation}\label{33}
\|u\|_{W_{2}^{1}}\leq C\|u\|_{W_{2}^{2}}^{1/2}\|u\|_{L_{2}}^{1/2}, \
\ \ u\in W_{2}^{1}(-1,0) \oplus W_{2}^{1}(0,1)
\end{equation}
Without loss of generality we shall assume that $\lambda=0$ is not
eigenvalue  of $\pounds_0.$ Otherwise we can find a real value $
\mu_0\notin\sigma(\pounds_0)$ and replace the spectral parameter
$\lambda$ by  $\lambda-\mu_0.$ Applying corollary \ref{selkff}  we
find that
\begin{equation}\label{34}
\|Au\|_{L_{2}}\leq C_1\|u\|_{W_{2}^{1}}\leq
C_2\|u\|_{W_{2}^{1}}^{1/2}\|u\|_{L_{2}}^{1/2}\leq
C_3\|\pounds_0u\|_{L_{2}}^{1/2}\|u\|_{L_{2}}^{1/2}
\end{equation}
for some $C_i=const (i=1,2,3)$ i.e. the operator  $A$ is
$\frac{1}{2}$-subordinate to $\pounds_0$.  Let us find asymptotic
behaviour of $N_{\pm}(r,\pounds_0)$ for $r \rightarrow \pm\infty$.
Consider the case  $p_1p_2<0$. Let $p_1<0, p_2>0$(the other case
$p_1>0, p_2<0$ is totaly similar).  The eigenvalues of the operator
$\pounds_0$ can be arranged as two infinite series
$\{\widetilde{\lambda}_{n,1}\}$ and $\{\widetilde{\lambda}_{n,2}\}$
with asymptotics
\begin{equation}\label{35}
\widetilde{\lambda}_{n,1}=-p_1\pi^{2} n^{2} + O(n), \ \
\widetilde{\lambda}_{n,2}=-p_2\pi^{2} n^{2} + O(n)
\end{equation}
(see, \cite{ka}). Then from (\ref{35}) we can derive easily that
\begin{eqnarray}\label{40}
N_{+}(r,\pounds_0)=\sum_{\widetilde{\lambda}_{n,1} \leq r}
1=\frac{1}{\sqrt{-p_1\pi}}\sqrt{r}+o(1) ,\ \ \ r \rightarrow \infty
\end{eqnarray}
and
\begin{eqnarray}\label{41}
N_{-}(r,\pounds_0)=\sum_{\widetilde{-\lambda}_{n,2} \leq r}
1=\frac{1}{\sqrt{p_2\pi}}\sqrt{r}+o(1) ,\ \ \ r \rightarrow \infty.
\end{eqnarray}
Further, applying Lemma \ref{27}  to the operators $\pounds_0$ and
$A$ we get that there is $b>0$ such that
\begin{eqnarray}\label{42}
|N_{\pm}(r,\alpha,\pounds)-N_{\pm}(r,\alpha,\pounds_0)|&\leq&
C_1(N_{\pm}(r+b\sqrt{r},\pounds_0)-N_{\pm}(r-b\sqrt{r},\pounds_0))\nonumber\\&+&C_2
\end{eqnarray}
for arbitrary $\alpha$  with $0<\alpha<\frac{\pi}{2}$, where $C_1$
and $C_2$  are some constants depending only on $\alpha$. Since
\begin{eqnarray}\label{43}
\sqrt{r+b\sqrt{r}}-\sqrt{r-b\sqrt{r}}=O(1), \ as \
r\rightarrow\infty
\end{eqnarray} from (\ref{41}) and (\ref{42}) it follows
that
\begin{eqnarray}\label{44}
(N_{\pm}(r+b\sqrt{r},\pounds_0)-N_{\pm}(r-b\sqrt{r},\pounds_0))\leq
C, \ \ r\rightarrow\infty
\end{eqnarray}
for some $C>0$. Hence, by virtue of (\ref{43}) for arbitrary
$\alpha$ with $0<\alpha<\frac{\pi}{2}$, there is a constant
$C_\alpha$ such that
\begin{eqnarray}\label{45}
|N_{\pm}(r,\alpha,\pounds)-N_{\pm}(r,\alpha,\pounds_0)|\leq C_\alpha
\end{eqnarray}
Taking in view  the fact that the spectrum of $\pounds_0$ is
discrete and using Corollaries \ref{selkf} and \ref{selkff} we have
that for all $\alpha$, $0<\alpha<\frac{\pi}{2}$, the number of
eigenvalues of $\pounds$ which are lying outside the angle
$\psi_{\alpha}^{\pm}=\{\lambda:|\arg(\pm\lambda)|<\alpha\}$ is
finite. Therefore from (\ref{41}), (\ref{42}) and  (\ref{43}) it
follows that
\begin{eqnarray}\label{46}
N_{+}(r,\frac{\pi}{2},\pounds)=\frac{1}{\sqrt{-p_1\pi}}\sqrt{r}+O(1)
,\ \ \ r \rightarrow \infty
\end{eqnarray}
and
\begin{eqnarray}\label{47}
N_{-}(r,\frac{\pi}{2},\pounds)=\frac{1}{\sqrt{p_2\pi}}\sqrt{r}+O(1)
,\ \ \ r \rightarrow \infty
\end{eqnarray}
Consequently in both left- and right half-plane  the operator
$\pounds$ has infinitely many eigenvalues. Denote by
$\{\lambda_{n,1}\}_{1}^{\infty}$ and
$\{\lambda_{n,2}\}_{1}^{\infty}$ all eigenvalues of operator
$\pounds$, which lies in the right and left half-plane respectively
and  arranged as $|\lambda_{1,i}|\leq|\lambda_{2,i}|\leq...$ (i=1,2)
according counted with their algebraic multiplicity. Then from
(\ref{46}) and (\ref{47}) we have
\begin{equation}\label{48}
|\lambda_{n,i}|=|p_i|\pi^{2} n^{2} + O(n), \ \ \ n \rightarrow
\infty (i=1,2)
\end{equation}
Further, by virtue of Lemma \ref{25}, there is $C>0$ such that
\begin{eqnarray}\label{49}
|Im \lambda_{n,i}|^{2}\leq C |\lambda_{n,i}|, \ (i=1,2)
\end{eqnarray}
and therefore for sufficiently large $n$(in fact, when
$|\lambda_{n,i}|\geq C$) we have
\begin{eqnarray}\label{50}
|Re \lambda_{n,i}|^{2}=|\lambda_{n,i}|^{2}-|Im
\lambda_{n,i}|^{2}\geq|\lambda_{n,i}|^{2}-C|\lambda_{n,i}|\geq(|\lambda_{n,i}|-C)^{2},
\ \ C=const
\end{eqnarray}
Consequently,
\begin{eqnarray}\label{51}
|Re \lambda_{n,i}|=|p_i|\pi^{2} n^{2} + O(n), \ \textrm{and }\ |Im
\lambda_{n,i}|=O(n)
\end{eqnarray}
i.e.
\begin{equation}\label{52}
\lambda_{n,i}=-p_i\pi^{2} n^{2} + O(n), \ \ (i=1,2)
\end{equation}
The proof for the case $p_1p_2>0$ is totally similar. The proof is
complete.
\end{proof}
\begin{theorem}\label{53} Let the condition 1. of Theorem \ref{30}
be satisfied and let the operator $A$ from $W_{2}^{2}(-1,0) \oplus
W_{2}^{2}(0,1)$ into $L_{2}(-1,0) \oplus L_{2}(0,1)$ acts compactly.

Case1. If $p_1p_2<0$ then the eigenvalues of the problem
(\ref{1})-(\ref{5}) can be arranged as two sequence
$\{\lambda_{n,1}\}_{1}^{\infty}$ and
$\{\lambda_{n,2}\}_{1}^{\infty}$ with asymptotics
\begin{equation}\label{54}
\lambda_{n,1}=-p_1\pi^{2} n^{2} + o(n^{2}), \ \
\lambda_{n,2}=-p_2\pi^{2} n^{2} + o(n^{2})
\end{equation}
Case2. If $p_1p_2>0$ then the eigenvalues of the problem
(\ref{1})-(\ref{5}) can be arranged  as one sequence
$\{\lambda_{n}\}_{1}^{\infty}$ with asymptotics
\begin{equation}\label{55}
\lambda_{n}= -\frac{p_1p_2}{p_1+2\sqrt{p_1p_2}+p_2}\pi^{2} n^{2} +
o(n^{2}). \ \
\end{equation}
\end{theorem}
\begin{proof}We are already shown in the proof the Theorem \ref{30},
that the operator $\pounds_0$ is self-adjoint with discrete spectrum
and for the eigenvalues of this operator the asymptotic formulas
(\ref{31})(for $p_1p_2<0$ ) and  (\ref{32}) (for $p_1p_2>0$) are
hold. At the other and, by virtue of the Corollary \ref{selkff} the
operator $R(\lambda,\pounds_0)$ is compact. Let us consider the case
$p_1<0$, $p_2>0$. From (\ref{40}) and (\ref{41}) it follows that
\begin{eqnarray}\label{56}
\lim_{\begin{array}{c}
        r \rightarrow \infty \\
        \varepsilon \rightarrow 0 \\
      \end{array}}
\ \frac{N_{\pm} \left(r(1 + \varepsilon),\pounds_0
\right)}{N_{\pm}(r, \pounds_0)}
 &=&1.
\end{eqnarray}
By virtue of the lemma \ref{27}, from (\ref{40}) and (\ref{41}) it
follows that
\begin{eqnarray}\label{57}
N_{\pm}(r,\alpha,\pounds)=N_{\pm}(r,\alpha,\pounds_0)+o(\sqrt{r}), \
r \rightarrow \infty
\end{eqnarray}
for all $\alpha$($0<\alpha<\frac{\pi}{2}$).Taking in view the
Corollary \ref{selkf} we see that the relation (\ref{57}) is
equivalent to the following asymptotic relation
\begin{equation}\label{58}
|\lambda_{n,i}|=|p_i|\pi^{2} n^{2} + O(n^{2}), \ \ \ n \rightarrow
\infty (i=1,2)
\end{equation}
Further from Corollary \ref{selkf} it follows that for all
$\alpha$($0<\alpha<\frac{\pi}{2}$) there is natural number
$n_{\alpha}$ such that for all $n\geq n_{\alpha}$
\begin{eqnarray}\label{59}
\frac{|Re \lambda_{n,i}|}{| \lambda_{n,i}|}>\cos\alpha, \ \
\frac{|Im \lambda_{n,i}|}{| \lambda_{n,i}|}<\sin\alpha \ \ (i=1,2)
\end{eqnarray}
Consequently
\begin{eqnarray}\label{60}
\lim\limits_{n \rightarrow \infty }\frac{|Re \lambda_{n,i}|}{|
\lambda_{n,i}|}=1, \ \ \lim\limits_{n \rightarrow \infty }\frac{|Im
\lambda_{n,i}|}{| \lambda_{n,i}|}=0 \ \ (i=1,2)
\end{eqnarray}
This means that
\begin{eqnarray}\label{61} |Re
\lambda_{n,i}|=|p_i|\pi^{2} n^{2} + O(n), \ |Im \lambda_{n,i}|=O(n)
\end{eqnarray}
i.e.
\begin{equation}\label{62}
\lambda_{n,i}=-p_i\pi^{2} n^{2} + O(n), \ \ (i=1,2)
\end{equation}
The proof is complete.
\end{proof}

\section{The Abel basis  of root functions of the problem
(\ref{1})-(\ref{5})} Let $\mathcal{H}$ be a separable Hilbert space
and $\mathcal{S}$ a unbounded closed linear operator acting in this
space with a dence domain $D(\mathcal{S})$. Assume that the spectrum
of $\mathcal{S}$ is discrete and $\{\lambda_j\}(j=1\div\infty)$ its
eigenvalues which arranged as $|\lambda_1|\leq|\lambda_2|\leq....$
Denote by $m_j$ the dimension of root lineal $M_{\lambda_{j}}$ and
let $f_1^{j}, f_2^{(j)},...,f_{m_{(j)}}^{(j)} $ be any orthonormal
basis of this root lineal. Let $\epsilon_j>0$ any real numbers, so
that $\epsilon_j<\min_{i\neq j}|\lambda_i-\lambda_j|$. Obviously the
contour $|\lambda_i-\lambda_j|=\epsilon_j$ surrounds only one
eigenvalue(namely the eigenvalue $\lambda_j$) It is known that (see,
\cite{kato}) the range of the projection operator
$P_{\lambda_{k}}(S)$ defined as
$$P_{\lambda_{k}}(S):=-\frac{1}{2\pi i}\oint\limits_{|\lambda-\lambda_k|=\epsilon_k}(\lambda I-\mathcal{S})^{-1}d\lambda$$
is contained in the root lineal $M_{\lambda_{j}}(S)$ and can be
represented as
$$P_{\lambda_{k}}(S)f=\sum_{i=1}^{m_k}c_i^{(k)}f_i^{(k)}$$
for each $f\in\mathcal{H}$. Under above assumptions the series(not
necessarily convergent)
$$f\sim\sum_{j=1}^{\infty}(-\frac{1}{2\pi i}\oint\limits_{|\lambda-\lambda_j|=\epsilon_j}(\lambda I-\mathcal{S})^{-1}fd\lambda)=\sum_{j=1}^{\infty}(\sum_{i=1}^{m_j}c_i^{(j)}f_i^{(j)})$$
is said to be a formal expansion of the vector $f\in\mathcal{H}$ in
the series of root vectors of $\mathcal{S}$. Let $\theta$ and
$\alpha$ any real positive numbers such that $\theta<\frac{\pi}{2}$
and $\alpha<\frac{\pi}{2\theta}$. Assume that the eigenvalues
$\lambda_j$(without, at least, finite number) of the operator
$\mathcal{S}$ are contained in the angle
$\chi_{\theta}=\{\lambda\in\mathbb{C}||arg\lambda|<\theta\}$. Then
for $\lambda^{\alpha}$ in this angle we mean
$\lambda^{\alpha}:=|\lambda|^{\alpha}e^{i\alpha arg\lambda}$.
Consequently for each constant $t>0$ the function
$|e^{-\lambda^{\alpha}t}|$ exponentially tends to zero in the angle
$\chi_{\theta}$ for $|\lambda|\rightarrow\infty$. If
$$\lim\limits_{t\rightarrow+0}\| f-\sum_{j=1}^{\infty}(-\frac{1}{2\pi i}\oint\limits_{|\lambda-\lambda_j|
=\epsilon_j}e^{-\lambda_{j}^{\alpha}t}(\lambda I-\mathcal{S})^{-1}fd\lambda)\|=0$$
then the system of root vectors of $\mathcal{S}$ is said to be an
Abel basis of order $\alpha$ where for
$\lambda_k\notin\chi_{\theta}$ the expression
$e^{-\lambda_{k}^{\alpha}t}$ is replace by $1$.
\begin{theorem}\label{dok}\cite{mar} If $\mathcal{S}$ is self-adjoint operator with
discrete spectrum in the Hilbert space and
\begin{equation}\label{kos}\liminf(N(r,R,\mathcal{S})/r^{s})<\infty
\end{equation} for some $s>0$ and if $B$ is p-subordinate to
$\mathcal{S}(0\leq p<1)$, then for each $\alpha>\max\{s-p+1,0\}$ the
system of root vectors of the operator $\mathcal{S}+B$ forms an Abel
basis of order $\alpha$ in the Hilbert space $\mathcal{H}$.
\end{theorem}
By using this theorem and the Theorem \ref{53} we shall prove the
next result.
\begin{theorem}Let the following conditions be satisfied:

1. $p_1\delta_{0}=p_2\gamma_{1}.$

2. The operator $A$ acted boundedly from the Hilbert space
$W_{2}^{1}(-1,0) \oplus W_{2}^{1}(0,1)$ into the Hilbert space
$L_{2}(-1,0) \oplus L_{2}(0,1)$.

3. Then the system of root functions of the main problem
(\ref{1})-(\ref{5}) forms an Abel basis of order $\alpha$ in the
Hilbert space $L_{2}(-1,0) \oplus L_{2}(0,1)$.
\end{theorem}
\begin{proof} Consider the case $p_1<0, p_2>0$(the other cases are
similar). Then from (\ref{40}) and (\ref{41}) it follows that the
condition (\ref{kos}) is satisfied for $s=\frac{1}{2}$.
Moreover,similarly to  the proof of the Theorem \ref{30} we can
prove that the operator $A$ is $\frac{1}{2}$ -subordinate to
$\pounds_0$. Consequently, it is enough to apply the Theorem
\ref{dok} to the operators $\pounds_0$ and $A$ to complete the
proof.
\end{proof}
\begin{remark}It is known that the property of a system of root
vectors to form an Abel basis of some order $\alpha>0$ is immediate
between the completeness of root vectors and a basis with
parentheses. Note that the concept of an Abel basis was first
introduced in \cite{yak}.
\end{remark}

\vskip 1truecm
\end{document}